\documentclass[reqno,a4paper,12pt]{amsart}
\usepackage{amsmath,amssymb,amsthm,amsfonts}
\usepackage[mathscr]{eucal}
\numberwithin{equation}{section}

\newcommand{\e}{\epsilon}
\newcommand{\p}{\psi}

\newcommand{\R}{{\mathbb R}}

\newcommand{\N}{{\mathbb{N}}}
\newcommand{\C}{{\mathbb{C}}}

\renewcommand{\qed}{\vrule height7pt width5pt depth0pt}

\newcommand{\Fc}{{\mathcal F}}
\newcommand{\Bc}{{\mathcal B}}
\newcommand{\Dc}{{\mathcal D}}
\newcommand{\Lc}{{\mathcal L}}
\newcommand{\Pc}{{\mathcal P}}
\newcommand{\Mcc}{{\mathcal M}}

\DeclareMathOperator{\Cp}{Cap}
\DeclareMathOperator{\Co}{Co}

\DeclareMathOperator{\supp}{supp}
\DeclareMathOperator{\rank}{rank}
\DeclareMathOperator{\Ran}{Ran}
\DeclareMathOperator{\Index}{index}
\DeclareMathOperator{\Ker}{Ker}
\DeclareMathOperator{\codim}{codim}

\newtheorem{theorem}{Theorem}[section]

\newtheorem{lemma}[theorem]{Lemma}
\newtheorem{corollary}[theorem]{Corollary}

\newtheorem*{theorem*}{Theorem}

\theoremstyle{definition}

\theoremstyle{remark}

\newtheorem*{remark*}{Remark}

\newcommand{\abs}[1]{\lvert#1\rvert}
\newcommand{\norm}[1]{\lVert#1\rVert}

\DeclareMathOperator{\Int}{Int}
\DeclareMathOperator{\Ext}{Ext}
\newcommand{\ess}{\text{\tiny{ess}}}

\begin{document}

\title[Bargmann-Toeplitz operators]{On the spectrum of Bargmann-Toeplitz operators with symbols of a variable sign}
\author[Pushnitski]{Alexander  Pushnitski}
\address[A. Pushnitski]{Department of Mathematics,
King's College London, Strand, London WC2R  2LS, U.K.}
\email{alexander.pushnitski@kcl.ac.uk}
\author[Rozenblum]{Grigori Rozenblum}
\address[G. Rozenblum]{1. Department of Mathematics \\
                          Chalmers University of Technology \\
                          2.Department of Mathematics  University of Gothenburg \\
                          Chalmers Tv\"argatan, 3, S-412 96
                           Gothenburg
                          Sweden}
\email{grigori@chalmers.se}

\begin{abstract}
The paper discusses
the spectrum 
of Toeplitz operators in  Bargmann spaces.
Our Toeplitz operators have real symbols with a variable 
sign and a compact support. 
A class of examples is considered where 
the asymptotics of the eigenvalues of such operators can 
be computed.
These examples show that this asymptotics 
depends on the geometry of the supports of the positive
and negative parts of the symbol.
Applications to the perturbed
Landau Hamiltonian are given.
\end{abstract}


\date{}

\maketitle


\section{Introduction}\label{intro}
\subsection{Bargmann-Toeplitz operators}
Let $\Fc^2$ be the Fock space, i.e. the Hilbert space of all entire
functions $f=f(z)$, $z\in\C$ such that
$$
\int_\C\abs{f(z)}^2e^{-\abs{z}^2}dm(z)
\equiv
\norm{f}^2_{\Fc^2}<\infty.
$$
Here and in what follows $dm(z)$ denotes the Lebesgue measure in $\C$.
Let $P_{\Fc^2}$ be the orthogonal projection in 
$L^2(\C,e^{-\abs{z}^2}dm(z))$ onto $\Fc^2$. 
For a bounded compactly supported function $V:\C\to\R$, 
the \emph{Bargmann-Toeplitz operator $T_V$ in $\Fc^2$} is the operator
\begin{equation}
T_V:\Fc^2\to\Fc^2, \quad T_Vf=P_{\Fc^2}Vf,\quad f\in\Fc^2.
\label{tb}
\end{equation}
The function $V$ is called the \emph{symbol} of the operator $T_V$.
The quadratic form of the operator $T_V$ is given by 
\begin{equation}\label{1:BargForm}
(T_V f,f)_{\Fc^2}=
\int_{\C}V(z)\abs{f(z)}^2 e^{-\abs{z}^2} d m(z), 
\quad f\in \Fc^2.
\end{equation}
It is well known that the operator $T_V$ is compact,
see e.g. \cite{Lue1,RaiWar}.
The subject of this paper is the rate of accumulation of the eigenvalues
of $T_V$ to zero for compactly supported symbols $V$ of a variable sign. 

For symbols $V$ with  a moderate decay at infinity, the asymptotics
of the eigenvalues of $T_V$ has been extensively studied in
the literature, see \cite{Raikov1,Raikov2,Sobolev,Tamura}
and also \cite[Chaps. 11 and 12]{Ivrii}.
For this class of symbols, there appears to be very little 
difference between the case of non-negative $V$ and 
the case of $V$ of a variable sign; this concerns both the 
results and the technique required for their 
proof. As it will be clear 
from the discussion below, the situation for compactly 
supported symbols $V$ is radically different.

We denote by  $\lambda_1^+\geq \lambda_2^+\geq\cdots>0$
the sequence of positive eigenvalues of $T_V$
enumerated with multiplicities taken into account.
If $T_V$ has only $k$ positive eigenvalues, we set $\lambda^+_n=0$
for $n\geq k+1$.
Similarly, let $\lambda_1^-\geq \lambda_2^-\geq\cdots>0$
be the sequence of positive eigenvalues
of $-T_V=T_{-V}$, appended by zeros if necessary. 
Finally,
$s_1\geq s_2\geq \cdots>0$ is the sequence of the singular values of $T_V$; 
of course, this is just the re-ordered union of the non-zero elements of the 
sequences $\{\lambda_n^+\}_{n=1}^\infty$ and $\{\lambda_n^-\}_{n=1}^\infty$.

\subsection{The case $V\geq0$}
In this case we have $T_V\geq0$ and so
$\lambda^-_n\equiv0$ and $\lambda_n^+=s_n$.
Let us first briefly discuss the results known for the case
of $V$ with a moderate decay at infinity. 
Suppose that 
\begin{equation}
V(x)=\mathbf v(x/\abs{x})\abs{x}^{-\rho}(1+o(1))
\quad\text{as $\abs{x}\to\infty$,}
\label{power}
\end{equation}
where $\rho>0$ is a constant 
and $\mathbf v\geq0$  is a function on the unit circle which 
satisfies some  regularity assumptions. 
Then one has \cite{Sobolev,Raikov1} $s_n=Cn^{-\rho/2}(1+o(1))$, 
as $n\to\infty$, where the constant $C$ can be explicitly expressed 
in terms  of $\mathbf v$ and $\rho$. 

Now suppose that $V$ is compactly supported. Then the above
result, in combination with  variational considerations, 
says only that $s_n$ decays faster than any power of 
$n$ as $n\to\infty$.
In fact, 
it turns out that the rate of decay of $s_n$ as $n\to\infty$ is
super-exponential.
In \cite{RaiWar,MelRoz} it was proved that, if compactly supported $V\geq0$ 
is separated from zero
on some open set (which is true, for example, 
if $V$ is continuous and not identically zero), then
\begin{equation}
\label{1:RWestimate log}
\log s_n=-n\log n +O(n),\quad  n\to\infty,
\end{equation}
or, using a somewhat informal but more expressive notation,
\begin{equation}
\label{1:RWestimate}
s_n=e^{-n\log n+O(n)}, \quad  n\to\infty.
\end{equation}
More precise asymptotics of $s_n$ is known \cite{FilPush}; we will
say more about this in Section~\ref{sec2.2}.

For future reference we would like to display a corollary of
\eqref{1:RWestimate log},
\eqref{1:RWestimate} which concerns the eigenvalue counting function.
For a self-adjoint operator $A$ and an interval $\Delta\subset\R$,
we denote by $E(\Delta; A)$ the spectral projection of $A$ corresponding to $\Delta$ and
by  $N(\Delta;A)=\rank E(\Delta;A)$ the total number of eigenvalues (counting multiplicities)
of $A$ in $\Delta$.
With this notation, \eqref{1:RWestimate} yields the asymptotic relation
\begin{equation}
N((\lambda,\infty);T_V)
=
\frac{\abs{\log\lambda}}{\log\abs{\log\lambda}}
(1+o(1)), \quad \lambda\to+0,
\quad V\geq0.
\label{a1}
\end{equation}
Note that  \eqref{a1} follows from \eqref{1:RWestimate} but the converse 
statement is false.

\subsection{The case of $V$ of a variable sign}
The main focus of this paper is the case of symbols
$V$ of a variable sign. We denote $V_\pm=(\abs{V}\pm V)/2\geq0$.
If $V$ has a power asymptotics at infinity \eqref{power} 
with $\mathbf v$ of a variable sign, then 
one has \cite{Sobolev,Raikov1} 
$\lambda^\pm_n=C^\pm n^{-\rho/2}(1+o(1))$, 
as $n\to\infty$, where the constants 
$C^\pm$ can be explicitly  expressed in terms of 
$\mathbf{v_\pm}$ and $\rho$. 
Essentially, in this case the operators $T_{V_+}$ and $T_{V_-}$ 
are ``asymptotically orthogonal''. This means that the asymptotics
of $\lambda_n^+$ (resp. $\lambda_n^-$)
is determined by the positive (resp. negative) part of the symbol $V$ only.
See also \cite{BS} for some older results of a similar nature.

Now suppose that the symbol $V$ is compactly supported. 
As we shall see, $T_{V_+}$ and $T_{V_-}$  are no longer 
asymptotically orthogonal in this case and the question of 
 the asymptotics of $T_V$ becomes more complicated. 
Of course, by the estimates $-T_{V_-}\leq T_V\leq T_{V_+}$,  the min-max
principle shows that we still have the super-exponential upper bounds
of the type
\eqref{1:RWestimate}
on both positive and negative eigenvalues of $T_V$:
$$
\lambda_n^\pm\leq e^{-n\log n+O(n)}, \quad n\to\infty.
$$
However, if both $V_+$ and $V_-$ are non-zero, the lower  bounds on the eigenvalues
of $T_V$ become a delicate problem because of the possible cancellations
between the contributions of $V_+$ and $V_-$.
Even the question of whether there are infinitely many non-zero
eigenvalues of $T_V$ in this case is non-trivial.
It was only in 2008 that the following theorem has been established
by D.~Luecking \cite{Lue2}:
\begin{theorem*}[\cite{Lue2}]
If $T_V$ is an operator of a finite rank, then $V\equiv0$.
\end{theorem*}
Thus, for $V\not\equiv0$ there are always infinitely many non-zero eigenvalues.
The purpose of this paper is to show by means of two simple ``extreme'' 
examples that the 
asymptotic behaviour of the eigenvalues $\lambda_n^\pm$ as $n\to\infty$
is determined by the geometry of the sets $\supp V_+$ and $\supp V_-$. 
In Section~\ref{qualitative} we show that if $\supp V_-$ is ``locked inside''
$\supp V_+$, then the number of negative eigenvalues of $T_V$ 
is finite. We prove a simple
\begin{theorem}\label{2:PositivePeriph}
Suppose that $V$ is
continuous and $K=\supp(V)$ is connected. Suppose also
that there exists a smooth closed simple curve $\Gamma\subset K$ such
that $V(x)\geq\delta>0$ on $\Gamma$ and
$\supp V_-$ lies inside the curve $\Gamma$.
Then $V$ has only finitely many negative eigenvalues.
\end{theorem}
Of course, by Luecking's theorem, the number of positive
eigenvalues in this case is infinite.
In Section~\ref{qualitative} we prove Theorem~\ref{2:PositivePeriph} 
and also 
discuss sufficient conditions on $V$ which ensure that $T_V$
has infinitely many negative eigenvalues.

At the other extreme, in Section~\ref{Sect.3} we consider a simple
class of symbols $V$ with the sets $\supp V_+$ and $\supp V_-$ 
``well separated'' and show that in this case there are infinitely 
many positive eigenvalues and infinitely many negative eigenvalues. 
We also establish some asymptotic bounds on the sequences
$\lambda_n^\pm$ in this case, see next subsection.

Finally, in Section~\ref{sect.4} we briefly discuss applications of the above results
 to the perturbed Landau Hamiltonian, i.e.
to the two-dimensional Schr\"odinger operator with a constant homogeneous magnetic field.
This was, in fact, the initial motivation for this work.

\subsection{Asymptotics bounds on the spectrum of $T_V$}
Here we discuss  the results concerning the case of $V$ with 
the supports of $V_+$ and $V_-$ being ``well separated''. 
Let $\Omega_+$ and $\Omega_-$ be two compact domains in $\C$
with Lipschitz boundaries.
We denote by $\Co(\Omega_\pm)$ the closed convex hull of $\Omega_\pm$, and we assume that
 \begin{equation}\label{1:disjoint}
    \Co(\Omega_+)\cap\Co(\Omega_-)=\varnothing.
 \end{equation}
 We consider the symbols $V$ of the form
 \begin{equation}\label{1:symbol V}
    V(z)=v_+(z)\chi_{\Omega_+}(z)-v_-(z)\chi_{\Omega_-}(z),
 \end{equation}
 where $v_\pm$ are bounded non-negative functions 
 such that $\inf_{\Omega_\pm}v_\pm>0$.
 In Section~\ref{Sect.3} we establish our main result:
 \begin{theorem}\label{1:theorem1.1}
 Let $\Omega_+$, $\Omega_-$ be Lipschitz domains in $\C$ which
satisfy \eqref{1:disjoint} and
let $V$ be of the form \eqref{1:symbol V}.
Then for the singular values $s_n$ of $T_V$ the asymptotics
\eqref{1:RWestimate} holds true.
Moreover, there exist constants $0< \delta_\pm\leq\Delta_\pm<1$
with $\delta_++\Delta_-=1$ and $\delta_-+\Delta_+=1$
such that  the asymptotic estimates
\begin{equation}\label{1:Theor1.1pm}
\exp\biggl(-\frac{n\log n}{\delta_\pm}+o(n\log n)\biggr)
\leq
\lambda_n^{\pm}
\leq
\exp\biggl(-\frac{n\log n}{\Delta_\pm}+o(n\log n)\biggr),
\quad n\to\infty
\end{equation}
hold true.
\end{theorem}
Theorem~\ref{1:theorem1.1}
yields the following estimates
for the eigenvalue counting functions:
\begin{gather}\label{1:RWN}
N((\lambda,\infty);\abs{T_V})\equiv N((\lambda,\infty);T_V)+N((\lambda,\infty);-T_V)
=
\frac{\abs{\log\lambda}}{\log\abs{\log\lambda}}
(1+o(1)), \quad \lambda\to+0,
\\\label{1:RWNpm}
\delta_\pm
\frac{\abs{\log\lambda}}{\log\abs{\log\lambda}}
(1+o(1))
\leq
N((\lambda,\infty);\pm T_V)
\leq
\Delta_\pm
\frac{\abs{\log\lambda}}{\log\abs{\log\lambda}}
(1+o(1)),
\quad \lambda\to+0.
\end{gather}
It would be interesting to study the optimal constants $\delta_\pm$,
$\Delta_\pm$ and
analyse their dependence on the pair of domains $\Omega_+$, $\Omega_-$.
Unfortunately, we do not have anything to say about this
except in
some very special case. Assume that $\Omega_+$ and
$\Omega_-$ can be interchanged by a Euclidean
motion in $\C$. More precisely,
suppose that for some $\theta\in \R$ and $\mathbf c\in\C$,
we can set $\varphi(z)=e^{i\theta}z+\mathbf c$ or 
$\varphi(z)=e^{i\theta}\bar{z}+\mathbf c$
and get $\varphi(\Omega_+)=\Omega_-$ and 
$\varphi(\Omega_-)=\Omega_+$.
\begin{theorem}\label{1:cor1.2}
Assume  the hypotheses of Theorem \ref{1:theorem1.1}
and suppose that $\Omega_\pm$ possess the above described symmetry. Then
$\delta_+=\delta_-=\Delta_+=\Delta_-=1/2$ in
\eqref{1:Theor1.1pm}. Moreover, the remainder estimate in
\eqref{1:Theor1.1pm} can be improved; one has
\begin{equation}\label{1:Cor.1}
\lambda^{\pm}_n=e^{-2n\log n+O(n)}, \quad n\to\infty.
\end{equation}
\end{theorem}

In our construction, the constants
$\delta_\pm$, $\Delta_\pm$ arise through the spectral analysis
of some auxiliary operators related to the pair of domains 
$\Omega_+$, $\Omega_-$, see Section~\ref{sec.b1}.
By analogy with the embedding operators, we call them 
the ``outbedding operators''.
The study of the spectral properties of these ``outbedding operators'' 
seems to be an interesting problem on its own.

\subsection{Notation}
For two domains $\Omega_\pm$ we denote by $\|\cdot\|_{\pm}$ and $(\cdot,\cdot)_{\pm}$
the norm and inner product in $L^2(\Omega_\pm,dm(z))$.
For a domain $\Omega\subset \C$, the set $\Omega^c$ is the 
complement of $\Omega$ in $\C$. 
Notation $\sigma_\ess(A)$ stands for the essential spectrum of 
self-adjoint operator $A$. 
For a closed simple Jordan curve $\Gamma\in\C$, we denote by
$\Int \Gamma$ (resp. $\Ext\Gamma$)
the bounded (resp. unbounded) component of $\C\setminus \Gamma$.
Finally, for a symbol $V$, we denote 
$W(z)=V(z)e^{-\abs{z}^2}$ and $V_\pm=(\abs{V}\pm V)/2$.

\section{Does $T_V$ have infinitely many negative eigenvalues?}\label{qualitative}

Here we address the following question. Let $V=V_+-V_-$ be a bounded compactly
supported function and suppose that neither $V_+$ nor $V_-$ is identically equal to zero.
What are sufficient conditions for $T_V$ to have infinitely many negative eigenvalues?
Of course, since $T_{-V}=-T_V$, this is equivalent to
the same question about positive eigenvalues.
The results presented here are by no means sharp; our purpose 
is  merely to illustrate the fact that the finiteness of the number of negative eigenvalues of $T_V$
depends upon the geometry of the sets $\supp V_+$ and $\supp V_-$.
\subsection{A ``negative'' example: 
the proof of Theorem~\ref{2:PositivePeriph}}
Our reasoning is based on the formula
\begin{equation}
(T_Vf,f)_{\Fc^2}
=
(T_{V_+}f,f)_{\Fc^2}
\left(1-\frac{(T_{V_-}f,f)_{\Fc^2}}{(T_{V_+}f,f)_{\Fc^2}}\right),
\quad 
f\in\Fc^2
\label{b1}
\end{equation}
and on an estimate for the quotient $(T_{V_-}f,f)_{\Fc^2}/(T_{V_+}f,f)_{\Fc^2}$.
For  any $f\in\Fc^2$ we have:
\begin{equation}
(T_{V_+}f,f)_{\Fc^2}
=
\int_{\supp V_+} W_+(z)\abs{f(z)}^2 dm(z)
\geq
\int_{\Ext \Gamma} W_+(z)\abs{f(z)}^2 dm(z).
\label{b2a}
\end{equation}
Next, let $\Gamma_\delta\subset\Ext\Gamma$ be a strip
along $\Gamma$ such that $W_+\geq\delta/2$ in $\Gamma_\delta$. 
Then, for some $C>0$, we have
\begin{equation}
\int_{\Ext \Gamma} W_+(z)\abs{f(z)}^2 dm(z)
\geq
\frac{\delta}{2}
\int_{\Gamma_\delta}\abs{f(z)}^2 dm(z)
\geq
C\int_\Gamma \abs{f(z)}^2 d\ell(z)
\label{b2}
\end{equation}
where $d\ell(z)$ is the length measure on the curve $\Gamma$.
Next, the Cauchy integral formula can be written as
$$
f(z)=\int_\Gamma K(z,\zeta)f(\zeta)d\ell(\zeta),
\quad
z\in\supp V_-,
$$
where $K(z,\zeta)$ is smooth and bounded for $z\in\supp V_-$,
$\zeta\in\Gamma$.
Thus, the integral operator from $L^2(\Gamma,d\ell)$ to
$L^2(\supp V_-, W_-(z)dm(z))$ with the kernel $K(z,\zeta)$
is compact.
It follows that for any $\varepsilon>0$ there exists a subspace
$\mathcal L_\varepsilon\subset \Fc^2$
of a finite codimension such that for all $f\in\mathcal L_\varepsilon$
one has
$$
(T_{V_-}f,f)_{\Fc^2}
=
\int_{\supp V_-} W_-(z) \abs{f(z)}^2 dm(z)
\leq
\varepsilon
\int_\Gamma \abs{f(z)}^2 d\ell(z).
$$
Recalling \eqref{b1}, \eqref{b2a}, \eqref{b2} and choosing $\varepsilon<\frac2{C}$,
we get
$$
(T_Vf,f)_{\Fc^2}\geq \frac12 (T_{V_+}f,f)_{\Fc^2}>0,
\quad \forall f\in\mathcal L_\varepsilon.
$$
Since the codimension of $\mathcal L_\varepsilon$ is finite,
by the min-max principle it follows that there are at most
finitely many negative eigenvalues of $T_V$.
\qed

\subsection{Sufficient conditions:  logarithmic capacity}\label{sec2.2}
This section is based
on sharp  asymptotic eigenvalue
estimates for Bargmann-Toeplitz operators with
non-negative weight obtained in \cite{FilPush}.
These estimates involve logarithmic capacity of the support of the symbol.
The notion of the logarithmic capacity of a compact set in $\C$  is introduced in
the framework of potential theory;
for a detailed exposition, see
e.g. \cite{Lan}. 
Let us recall some basic properties of
logarithmic capacity.

(i) If $K_1\subset K_2\subset\C$, then $\Cp K_1\leq \Cp K_2$.

(ii) For a compact set $K\subset\C$, $\Cp K$ coincides
with the logarithmic capacity of the outer boundary of
$K$ (i.e. the boundary of the unbounded component of 
$K^c$).

(iii) The logarithmic capacity of a disc of the radius $R$ is $R$.

We denote
\begin{equation*}\label{5:internal}
      \supp_-(V)=
\left\{z\in \C\mid\limsup_{r\to+0}
\frac{\log\int_{\abs{\zeta-z}\leq r}\abs{V(\zeta)}dm(\zeta)}{\log
r}<\infty\right\}.
\end{equation*}
It is clear that $\supp_-(V)\subset \supp (V)$, and all the
Lebesgue points $z$ of $V$ such that $V(z)>0$ belong to
$\supp_-(V)$.

\begin{theorem}\label{5:Teo2}
Suppose that the symbol $V=V_+-V_-$ with
positive and negative parts $V_+,V_-\ge0$ satisfies
\begin{equation}\label{5:Teo2.0}
      \Cp \supp_-(V_-)>\Cp \supp( V_+) .
\end{equation}
Then the operator $T_V$ has infinitely many
negative eigenvalues.
\end{theorem}

We note that by the properties (i), (ii) of logarithmic capacity,
under the hypothesis of
Theorem~\ref{2:PositivePeriph}
we necessarily have $\Cp\supp_-(V_-)\leq\Cp\supp(V_+)$.
Thus, Theorem~\ref{2:PositivePeriph} and
Theorem~\ref{5:Teo2} are in agreement with
each other.

In order to prove Theorem~\ref{5:Teo2}, we
 need to recall the results of \cite{FilPush} concerning the
behavior of the spectrum of the Bargmann-Toeplitz
operators $T_V$ with non-negative  compactly supported
symbols $V$.
Denote by $s_n=s_n(T_V)$ the singular numbers of
the operator $T_V$ for $V\geq0$.  
Lemma~1 of \cite{FilPush} together with the estimates (1.15), (1.16) 
of \cite{FilPush} and the Stirling formula provides the estimates

\begin{gather}\label{5:FiPu2}
\!\!\!e(\Cp \supp_-(V))^2
\le
\liminf_{n\to\infty}{n}s_n^{1/n}
\le
\limsup_{n\to\infty} {n} s_n^{1/n}\le e(\Cp
\supp(V))^2.
\end{gather}

\begin{proof}[Proof of Theorem~\ref{5:Teo2}]
In order to simplify our notation, 
we will consider $T_{-V}$ instead of $T_V$.
That is, we assume that $\Cp \supp_-(V_+)>\Cp\supp(V_-)$
and prove that $T_V$ has infinitely many positive eigenvalues.
Take
$\varepsilon>0$ such that
$$
a=e(\Cp \supp_-(V_+))^2-\varepsilon>b
=
e(\Cp \supp(V_-))^2+\varepsilon.
$$
It follows from \eqref{5:FiPu2} that for $m,n$
sufficiently large,
\begin{equation}\label{5:Teo2.1}
\lambda^+_n(T_{V_+})\ge n^{-n}a^n,\quad
\lambda^+_m(T_{V_-})\le m^{-m}b^m.
\end{equation}
By the well known inequality for the eigenvalues 
of a sum of two compact operators (see e.g. \cite[Section 95, formula (10)]{RN})
we have
$$
\lambda^+_{n+m-1}(T_{V_+})
=
\lambda^+_{n+m-1}(T_V+T_{V_-})
\leq
\lambda_n^+(T_V)+\lambda_m(T_{V_-}),
$$
which can be rewritten as
\begin{equation}
\label{b3}
\lambda_{n}^+(T_V)\ge \lambda^+_{m+n-1}(T_{V_+})-\lambda^+_{m}(T_{V_-}).
\end{equation}
Below we choose an increasing sequence of integers $n_1$, $n_2$,\dots
such that for each $n=n_m$ the r.h.s. of \eqref{b3} is positive.
This will prove that $T_V$ has infinitely many positive eigenvalues.

By \eqref{5:Teo2.1}, it suffices to prove that for each $m$ there
exists $n$ such that
\begin{equation}\label{5:Teo2.3}
(m+n-1)^{-(m+n-1)}a^{m+n-1}>m^{-m}b^m .
\end{equation}
An elementary analysis shows that if $m$ is sufficiently
large, then any $n<\gamma m/\log m$ with $\gamma<\log(a/b)$
satisfies \eqref{5:Teo2.3}.
This proves the required statement.
\end{proof}

\subsection{Application of conformal mapping}

Theorem~\ref{2:PositivePeriph} shows that if $\supp V_-$ is
``locked'' inside $\supp V_+$, then $T_V$ has finitely many  negative
eigenvalues. Here we show that in some sense the converse
is true: if $V$ is negative somewhere at the outer boundary
of its support, then $T_V$ has infinitely many negative eigenvalues.

We state the following result in its simplest form; it is easy to suggest
various generalisations. 
\begin{theorem}\label{Theorem2G}
Suppose that $V$ is continuous and the outer boundary of $\supp V$ is a smooth
simple curve $\Gamma$. Suppose that for some $z_0\in\Gamma$
and some $r>0$, one has
$$
V(z)<0 \text{ for all $z\in\Int \Gamma$, $\abs{z-z_0}<r$. }
$$
Then $T_V$ has infinitely many negative eigenvalues.
\end{theorem}

\begin{proof}[Sketch of proof]
The idea of the proof is to apply a conformal mapping which
``blows up'' the support of $V_-$ and then to use
Theorem~\ref{5:Teo2}. 
The proof also uses
the reformulation of the problem in terms of Toeplitz
operators in Bergman spaces. 
Without going into details, we describe
the main steps of the proof.

\textbf{Step 1:}
Let $\Dc$ be a simply connected domain in $\C$ with smooth boundary. 
The Bergman space $\Bc^2(\Dc)$ is the subspace of 
$L^2(\Dc)=L^2(\Dc, dm)$ which consists of all functions analytic in $\Dc$. 
Let $P_{\Bc^2}$ be the orthogonal projection in $L^2(\Dc)$ onto $\Bc^2(\Dc)$. 
Then, for a bounded function $F$ with $\supp F\subset \Dc$, the 
Bergman-Toeplitz operator $T_F(\Dc)$ is the operator 
$$
T_F(\Dc):\Bc^2(\Dc)\to\Bc^2(\Dc), 
\quad T_F(\Dc) g=P_{\Bc^2} Fg, 
\quad g\in\Bc^2(\Dc).
$$
The quadratic form of the operator $T_F(\Dc)$ is given by 
$$
(T_F(\Dc)g,g)_{\Bc^2(\Dc)}
=
(Fg,g)_{L^2(\Dc)},
\quad 
g\in \Bc^2(\Dc).
$$ 
We use the following important observation 
(see \cite{Roz} for the details).
Since the set of all analytic polynomials is dense in $ \Bc^2(\Dc)$,
by the min-max principle it follows that 
the negative spectrum of $T_F(\Dc)$ is infinite if and only if
for any $n\in\N$ there exists a linear set $\Lc_n$, $\dim \Lc_n=n$
of analytic polynomials such that $(Fg,g)_{L^2(\Dc)}<0$ 
for all $g\in\Lc_n$, $g\not=0$. 
Note that for any polynomial $g$ the value of the quadratic
form $(Fg,g)_{L^2(\Dc)}$ depends on $F$ but not on the choice
of the domain $\Dc$ as long as $\supp F\subset \Dc$. 
Thus, the above stated necessary and sufficient condition 
of the infiniteness of the negative spectrum of $T_F(\Dc)$ 
is in fact independent of $\Dc$. 

This reasoning can also be applied 
to the Bargmann-Toeplitz operator $T_V$.
We arrive at the following conclusion (recall that  $W(z)=V(z)e^{-\abs{z}^2}$):
\emph{the Bargmann-Toeplitz operator $T_V$ in $\Fc^2$ has infinitely many 
negative eigenvalues if and only if for some domain 
$\Dc\supset\supp V$ the Bergman-Toeplitz operator $T_W(\Dc)$ 
in $\Bc^2(\Dc)$ has infinitely many negative eigenvalues.}

\textbf{Step 2:}
Under the hypothesis of Theorem~\ref{Theorem2G}, let us fix a point $z_1$ 
outside the support of $V$ and choose a simply connected domain $\Dc$ 
with smooth boundary such that $\supp V\subset \Dc$ 
and $z_1\notin\overline{\Dc}$
(here $\overline{\Dc}$ is the closure of $\Dc$).
Consider the conformal map $w=\varphi(z)=(z-z_1)^{-1}$
and let $\Dc'=\varphi(\Dc)$. 
The map $\varphi$ generates a unitary map $\Phi$ 
from $L^2(\Dc)$ to $L^2(\Dc')$: 
$$
\Phi: f(z)\mapsto (\Phi f)(w)=f(\varphi^{-1}(w))\abs{\varphi'(\varphi^{-1}(w))}^{-1/2}.
$$
We have $\Phi T_W(\Dc) \Phi^*=T_{W'}(\Dc')$ 
with $W'=W\circ\varphi$.
Thus, the operator $T_W(\Dc)$ in $\Bc^2(\Dc)$ has infinitely many negative
eigenvalues if and only if the operator $T_{W'}(\Dc')$ in $\Bc^2(\Dc')$ 
has infinitely many negative eigenvalues.
Combining this with the previous step, we obtain that 
\emph{the operator $T_V$ in $\Fc^2$ has infinitely
many negative eigenvalues if and only if the operator
$T_{V'}$ in $\Fc^2$ has infinitely many negative eigenvalues, 
where $V'(z)=e^{\abs{z}^2}W'(z)$.}

\textbf{Step 3:}
Now it remains to choose $z_1$ in such a way that 
\begin{equation}
\Cp \supp_- (V'_-)>\Cp \supp (V'_+);
\label{*}
\end{equation}
then Theorem~\ref{5:Teo2} can be applied to 
$T_{V'}$.
Clearly, we have $\supp (V'_+)=\varphi(\supp V_+)$.
Next, let $\Omega_-=\{z\in\C\mid V(z)<0\}$; 
by the continuity of $V$, we have 
$$
\Omega_-\subset \supp_-(V_-)
\quad \text{ and } \quad
\varphi(\Omega_-)\subset\supp_-(V'_-).
$$
Now it is easy to see that by choosing $z_1$ sufficiently close to $z_0$, we can 
ensure that 
$$
\Cp \varphi(\Omega_-)>\Cp\varphi(\supp V_+),
$$ 
and thus \eqref{*} holds true.
It follows that $T_{V'}$, and therefore $T_V$ 
has infinitely many negative eigenvalues. 
\end{proof}

\section{Symbols with separated positive and negative parts}\label{Sect.3}

In this section we prove Theorems~\ref{1:theorem1.1} and \ref{1:cor1.2}.
\subsection{``Outbedding'' operators and $A_n^{\pm}$}\label{sec.b1}
Throughout this section, we assume that $\Omega_+$ and $\Omega_-$
are Lipschitz domains which satisfy \eqref{1:disjoint}.
We start with some notation.
For $n\in\N$ we denote by $\Pc_n$ the space of all polynomials in $z$
of degree $\le n$.
The $(n+1)$-dimensional linear space $\Pc_n$ can be endowed with
the inner product structure of $L^2(\Omega_\pm)$.  We denote the resulting
(finite dimensional) Hilbert spaces by $\Pc^{\pm}_n$.
Consider the operators
$$
S_n^+:\Pc^+_n\to \Pc^-_n, \quad S_n^+f=f
\quad\text{ and }\quad
S_n^-:\Pc^-_n\to \Pc^+_n, \quad S_n^-f= f.
$$
Thus, $S_n^+$, $S_n^-$ map each polynomial to itself.
By analogy with embedding operators, we will call
$S_n^+$, $S_n^-$ 
\emph{the outbedding operators.} 
We define the operators $A_n^{\pm} $ in $\Pc^{\pm}_n$  as
 $$A_n^{\pm}= (S^{\pm}_n)^*S^{\pm}_n.$$
The quadratic  forms of these operators are
$$
(A_n^{+}f,f)_{+}=\norm{f}^2_{-}, \quad  f\in \Pc_n^+
\quad \text{ and }\quad
(A_n^{-}f,f)_{-}=\norm{f}^2_{+}, \quad f\in \Pc_n^-.
$$
By definition, we have
$$
\frac{(A_n^+f,f)_+}{\|f\|^2_+}
=
\frac{\|f\|_-^2}{(A_n^-f,f)_-},
\quad f\in\Pc_n,
$$
and therefore, by the min-max principle,
$$
\dim\Ker(A_n^+-\lambda I)=\dim\Ker(A_n^- -\lambda^{-1} I), \quad \lambda>0.
$$
It follows that
\begin{equation}\label{3.1}
N((\lambda,\infty);A_n^\pm)=N((0,\lambda^{-1});A_n^{\mp}), \quad \lambda>0.
\end{equation}

In our discussion of the spectrum of Toeplitz operators,
the distribution of singular values
of the outbedding operators
plays a  crucial role.
\begin{lemma}\label{Lemma3.1}
For any $0<a_1<a_2<\infty$,
\begin{equation}\label{3.2}
    N((a_1,a_2);A_n^{\pm})=O(1)\ {\mathrm{ as }}\ n\to\infty.
\end{equation}
\end{lemma}
\begin{proof}
1.
Let us choose disjoint closed sets $\Omega'_+$ and $\Omega'_-$
whose boundaries are simple Jordan curves 
and such that $\Omega_\pm\subset\Omega'_\pm$. 
Consider the function $\phi(z)$ defined on $\Omega'_+\cup\Omega'_-$,
such that
$\phi(z)=1$ for $z\in\Omega'_+$ and $\phi(z)=0$ for $z\in\Omega'_-$.
By Theorem 6, Chapter 2 in \cite{Walsh},
the function $\phi$ can be approximated by polynomials uniformly on 
$\Omega'_+\cup\Omega'_-$.
In particular, for any $\gamma>0$ there exists a polynomial 
$p\not\equiv0$ such that
\begin{equation}\label{3.6}
\sup_{\Omega_-}|p(z)|
\leq 
\gamma \inf_{\Omega_+}|p(z)|.
 \end{equation}
Let us fix a polynomial $p$ which satisfies  \eqref{3.6}
with $\gamma^2=a_1/a_2$ and set $m=\deg p$.

2.
Consider the subspace $\Lc_n=\Ran E([0,a_2],{A_n^+})\subset\Pc_n^+$. Then
\begin{equation}\label{3.7}
\|f\|^2_-\le a_2\|f\|^2_+,\quad  \forall f\in \Lc_n.
\end{equation}
Denote $\Mcc_{n+m}=\{pf \ : f\in\Lc_n\}\subset\Pc_{n+m}^+$.
Then for any $g\in \Mcc_{m+n}$, we have
$$
\frac{\|g\|_-^2}{\|g\|_+^2}
=
\frac{\|pf\|_-^2}{\|pf\|_+^2}
\le
\frac{\|f\|_-^2\sup_{\Omega_-}|p|^2}{\|f\|_+^2\inf_{\Omega_+}|p|^2}
\leq
\gamma^2 a_2
=a_1.
$$
It follows that
\begin{equation}\label{3.8}
N([0,a_1];A^+_{m+n})\ge \dim \Mcc_{n+m}
=
\dim \Lc_n=N([0,a_2];A_n^+).
\end{equation}

3.
Since $\dim \Pc_n^+=n+1$, the inequality \eqref{3.8} can be written as
  \begin{equation}\label{3.9}
    (n+m+1)-N((a_1,\infty);A^+_{n+m})\ge (n+1)-N((a_2,\infty);A^+_{n}).
  \end{equation}
  By variational considerations, we have
  $$N((a_1,\infty);A_n^+)\le N((a_1,\infty);A_{n+m}^+).$$
  Combining this with \eqref{3.9}, we get
  $$N((a_1,\infty);A_n^+)-N((a_2,\infty);A_n^+)\le m,$$
  which implies $N((a_1,a_2);{A_n^+})\le m$.
\end{proof}

In the next lemma we  estimate the norm of the outbedding operators
$S_n^{\pm}$ or, equivalently, the norm of the operators $A_n^\pm$.
In order to do this, and for our further analysis,
we need to recall some facts
from \cite{StTotik} related to  the theory of orthogonal polynomials.

Let $P_n^{\pm}$ be the $n$-th degree orthogonal polynomial
corresponding to the measure $\chi_{\Omega_\pm}(z)dm(z)$ in $\C$.
We assume that $P_n^\pm$ is normalised such that
$\norm{P_n^\pm}_\pm=1$.
Let $g_\pm$ be the Green's function corresponding to the domain
$\Omega_\pm^\mathrm{c}$; this function is uniquely defined by the 
requirements that $g_\pm$ is harmonic in $\Omega_\pm^c$, 
vanishes on $\Omega_\pm$ and $g_\pm(z)-\log\abs{z}=O(1)$ 
as $\abs{z}\to\infty$. 
We will use the following estimates from 
\cite[Theorem~1.1.4 and Lemma~1.1.7]{StTotik}:
\begin{equation}\label{3.3}
\limsup_{n\to\infty}|P_n^{\pm}(z)|^{1/n}
\le
e^{g_\pm(z)}
\end{equation}
locally uniformly in $\C$,
\begin{equation}\label{3.4}
\liminf_{n\to\infty}|P_n^\pm(z)|^{1/n}
\ge
e^{g_{\pm}(z)}
\end{equation}
locally uniformly in $\C\setminus \Co(\Omega_\pm)$.
We will denote
\begin{equation}
a_+=\sup_{\Omega_-} g_+, 
\qquad
a_-=\inf_{\Omega_+}g_-.
\label{3.5}
\end{equation}
\begin{lemma}\label{3:PropNorm}
Let $\Omega_+$ and $\Omega_-$ be
domains with Lipschitz boundaries such that
$\Co(\Omega_+)\cap\Co(\Omega_-)=\varnothing$.
 Then, for any $b_+>a_+$ one has
\begin{equation}\label{3.11}
\|S_n^+\|=\|A_n^+\|^{1/2}\le e^{b_+n}
\end{equation}
for  all sufficiently large $n$.
\end{lemma}
\begin{proof}
It follows from \eqref{3.3} that $\|P_n^+\|_-\le Ce^{b_+n}$
for some $C>0$ and all $n$. 
Let $q$ be a polynomial in $\Pc_n$.
We expand $q$  in the basis of orthogonal polynomials
$P_k^+, \ k=0,\dots,n:$  $q=\sum_{k=0}^{n}c_kP_k^+$.
Then  we have
\begin{gather*}
\|q\|_-^2
\leq
\left( \sum_{k=0}^n|c_k|\|P_k^+\|_-\right)^2
\leq
\max_{k=0,\dots,n} \norm{P_k^+}_-^2
\left(\sum_{k=0}^n\abs{c_k}\right)^2
\leq
C^2e^{2b_+n}\left(\sum_{k=0}^n\abs{c_k}\right)^2
\\
\leq
C^2e^{2b_+n}(n+1)\sum_{k=0}^n\abs{c_k}^2
=
C^2e^{2b_+n}(n+1)\norm{q}_+^2.
\end{gather*}
Thus, for all sufficiently large $n$ we have
$$
\norm{q}_-
\leq
Ce^{b_+n}\sqrt{n+1}\norm{q}_+
\leq
C_\varepsilon e^{(b_++\varepsilon)n}\norm{q}_+
$$
where $\varepsilon>0$ can be chosen arbitrary small.
This proves the required statement.
\end{proof}
\begin{lemma}\label{3:Prop.2.3}
Denote
$$
\delta_\pm=\liminf_{n\to\infty}\frac1n N((0,1);A_n^{\pm}),
\quad
\Delta_\pm=\limsup_{n\to\infty}\frac1n N((0,1);A_n^{\pm}).
$$
Then  we have
\begin{equation}\label{3.10}
0<\delta_\pm\leq \Delta_\pm<1
\end{equation}
and
\begin{equation}
 \delta_-+\Delta_+=1,\quad
\delta_++\Delta_-=1.
\label{3.10a}
\end{equation}
\end{lemma}

\begin{proof}

1. Let $a_+$, $a_-$ be as in \eqref{3.5} and 
let $b_\pm$ be positive constants such that $b_+>a_+$ and $b_-<a_-$.
Let $q$ be a polynomial with $\deg q\le k$.
Using \eqref{3.3} and \eqref{3.11},
we obtain that for any $\delta>0$ and for all sufficiently large $k$ and $m$,
\begin{equation}\label{3.12}
    \|P_m^-q\|_-\le e^{\delta m}\|q\|_-\le e^{\delta m+b_+k}\|q\|_+.
\end{equation}
On the other hand, by \eqref{3.4}, for
all sufficiently large $m$,
\begin{equation}\label{3.13}
\|P_m^- q\|_+\ge e^{b_-m}\|q\|_+.
\end{equation}
Combining \eqref{3.12} and \eqref{3.13}, we obtain that
for any $b_+>a_+$ and $b_-<a_-$ and 
all sufficiently large $m$ and $k$ we have
\begin{equation}\label{3.14}
    \|P_m^- q\|_-\le e^{b_+k-b_-m}\|P_m^- q\|_+;
\end{equation}
here $\delta$ has been incorporated into $b_-$. 

2.
Let us  choose $\alpha\in(\frac{a_+}{a_++a_-},1)$  and set $\beta=1-\alpha$.
Consider the subspace
$$
\Lc_n=\{P^-_{[\alpha n]}q:\ \deg  q< [\beta n] \}\subset\Pc_n,\ \dim \Lc_n=[\beta n].
$$
By \eqref{3.14}, we have for any $b_-<a_-$ and any $b_+>a_+$
and all sufficiently large $n$,
$$
\|f\|_-\le e^{b_+[\beta n]-b_-[\alpha n]}\|f\|_+, \ \forall f\in \Lc _n.
$$
We have $b_+[\beta n]-b_-[\alpha n]\le(b_+\beta-b_-\alpha)n+b_- $.
By our choice of $\alpha$ and $\beta$,  
we have $a_+\beta-a_-\alpha<0$, and therefore the expression $(b_+\beta-b_-\alpha)$ 
can be made negative
by a suitable choice of $b_\pm$. Thus we obtain that for all sufficiently large $n$,
\begin{equation}\label{3.15}
\|f\|_-< \|f\|_+, \ \forall f\in \Lc_n, \quad f\not=0.
\end{equation}

3.
From \eqref{3.15} it follows that $N((0,1);A_n^{+})\ge \dim \Lc_n=[\beta n],$ and so
$\delta_+\geq \beta>0$.
Interchanging $\Omega_+$ and $\Omega_-$, we also get
$\delta_->0$.
Using  \eqref{3.1} and Lemma~\ref{Lemma3.1}, we obtain
\begin{gather*}
0<\delta_- 
=
\liminf_{n\to\infty}n^{-1}N((0,1);A_n^-)
=
\liminf_{n\to\infty}n^{-1}N((1,\infty);A_n^+)
\\
=
\liminf_{n\to\infty}n^{-1}\bigl(n+1-N((0,1);A_n^+)-N(\{1\};A_n^+)\bigr)
=
1-\Delta_+,
\end{gather*}
and so $\Delta_+<1$. In a similar way, one proves that $\delta_+=1-\Delta_-$
and so $\Delta_-<1$.
\end{proof}

\subsection{Estimates for Toeplitz operators}
In order to prove Theorem \ref{1:theorem1.1},
we first establish some relations between the spectral
distribution functions of
$T_V$  and $A_n^\pm$.
We need two elementary auxiliary statements.
We denote
$D_r(a)=\{z\in\C\mid \abs{z-a}<r\}$ and 
$$
z^n\Fc^2=\{f\in\Fc^2\mid f(z)z^{-n}\in\Fc^2\},
\quad n\in\N.
$$
\begin{lemma}\label{lma3.3}
(i) 
For any $r>0$ there exists a constant $\alpha>0$ 
such that for all sufficiently large $n$, one has
$$
\norm{f}^2_{L^2(D_r(0))}
\leq 
e^{-n\log n+\alpha n}
{\|f\|^2_{\Fc^2}},
\quad 
\forall f\in z^n\Fc^2.
$$
(ii)
For any $a\in\C$ and any $r>0$ there exists a constant $\alpha>0$ 
such that for all sufficiently large $n$, 
$$
\norm{f}^2_{L^2(D_r(a))}
\geq 
e^{-n\log n-\alpha n}
{\|f\|^2_{\Fc^2}},
\quad 
\forall f\in \Pc_n.
$$
\end{lemma}
\begin{proof}
(i) Write $f\in z^n\Fc^2$ as
$f(z)=\sum_{j=n}^\infty f_jz^j$.  We get
\begin{gather*}
\|f\|^2_{L^2(D_r(0))} =\pi \sum^{\infty}_{j=n}|f_j|^2 \frac{r^{2j+2}}{j+1},
\\
\|f\|^2_{\Fc^2}
=
\sum_{j=n}^\infty\abs{f_j}^2\norm{z^j}_{\Fc^2}^2
=
\pi \sum^{\infty}_{j=n}|f_j|^2 j!
\end{gather*}
and therefore, for all sufficiently large $n$, 
\begin{equation*}
\|f\|^2_{L^2(D_r(0))}
= 
\pi \sum^{\infty}_{j=n}|f_j|^2 j! \frac{r^{2j+2}}{(j+1)!}
\leq 
\frac{r^{2n+2}}{(n+1)!}\|f\|^2_{\Fc^2}.
\end{equation*}
Combining this with the Stirling formula, we obtain 
the required estimate.

(ii)
Without the loss of generality assume $r\leq1$. 
Write $f\in\Pc_n$ as $f(z)=\sum_{j=0}^n f_j (z-a)^j$. 
Then 
\begin{gather*}
\norm{f}^2_{L^2(D_r(a))}
=
\pi\sum_{j=0}^n \abs{f_j}^2
 \frac{r^{2j+2}}{j+1},
\\
\norm{f}_{\Fc^2}
\leq
\sum_{j=0}^n\abs{f_j}\norm{(z-a)^j}_{\Fc^2}
\leq
\max_{k=0,\dots,n}\norm{(z-a)^k}_{\Fc^2}\sum_{j=0}^n \abs{f_j}.
\end{gather*}
An elementary calculation shows that 
$\norm{(z-a)^k}_{\Fc^2}^2\leq e^{\gamma k}k!$
for some $\gamma>0$ and all $k\in\N$. 
Thus, 
\begin{multline*}
\norm{f}_{\Fc^2}^2
\leq
e^{\gamma n}n!\left(\sum_{j=0}^n\abs{f_j}\right)^2
\leq
e^{\gamma n}n!
\left(\sum_{j=0}^n\abs{f_j}\frac{r^{j+1}}{\sqrt{j+1}}\right)^2
(n+1)r^{-2n-2}
\\
\leq
e^{\gamma n} n! \left(\sum_{j=0}^n \abs{f_j}^2\frac{r^{2j+2}}{j+1}\right)
(n+1)^2 r^{-2n-2}
\leq
e^{\gamma n}n!
\pi^{-1}\norm{f}^2_{L^2(D_r(a))}
(n+1)^2 r^{-2n-2}.
\end{multline*}
Combining this with the Stirling formula, 
we get the required statement.
\end{proof}

Recall that by our assumptions, the symbol $V$ is given 
by the formula \eqref{1:symbol V}, 
where the functions $v_{\pm}$ are bounded, positive, and 
separated from zero. 
Let us write 
$$
W(z)=W_+(z)-W_-(z)
=
w_+(z)\chi_{\Omega_+}(z)-w_-(z)\chi_{\Omega_-}(z),
$$ 
where the functions $w_\pm$ satisfy
\begin{equation}\label{3:symbol}
0<\tau_\pm \le w_\pm(z)\le \sigma_\pm
\end{equation}
with some constants $\tau_\pm$ and $\sigma_\pm$. 

\begin{lemma}\label{3.lemma31}
For all sufficiently large $\alpha$, all sufficiently small 
$\varkappa\in(0,1)$ and
all sufficiently large $n$, one has
\begin{equation}\label{4.1}
N((e^{-n\log n}e^{-\alpha n},\infty);T_V)
\geq N((0,\varkappa);{A_n^+}).
\end{equation}
\end{lemma}
\begin{proof}
Choose $\varkappa$ sufficiently small so that 
$\gamma=\frac{\sigma_-}{\tau_+}\varkappa<1$.
For all $f$ in the subspace $\Ran E((0,\varkappa);{A_n^+})\subset\Pc_n^+$
we have
$$
(W_-f,f)_-
\leq
\sigma_-\norm{f}_-^2
<
\sigma_- \varkappa \norm{f}_+^2
\leq
\frac{\sigma_-\varkappa}{\tau_+}(W_+f,f)_+
=
\gamma(W_+f,f)_+.
$$
Therefore
\begin{equation}
\frac{(T_V f,f)_{\Fc^2}}{\|f\|^2_{\Fc^2}}
=
\frac{(W_+f,f)_+-(W_-f,f)_-}{\|f\|^2_{\Fc^2}}
\geq
(1-\gamma)\frac{(W_+f,f)_+}{\|f\|^2_{\Fc^2}}
\geq
(1-\gamma)\tau_+\frac{\|f\|_+^2}{\|f\|^2_{\Fc^2}}
\label{4.1a}
\end{equation}
for all $f\in\Ran E((0,\varkappa);{A_n^+})$.
Choose a disk $D_r(a)\subset\Omega_+$; 
then $\norm{f}_+\geq\norm{f}_{L^2(D_r(a))}$.
Now combining \eqref{4.1a} with 
Lemma~\ref{lma3.3}(ii), we obtain
$$
(T_Vf,f)_{\Fc^2}
\geq 
e^{-n\log n}e^{-\alpha n}\norm{f}^2_{\Fc^2},
\quad 
\forall f\in\Ran E((0,\varkappa);A_n^+),
$$
for some constant $\alpha>0$ and all sufficiently large $n$.
Applying the min-max principle, we obtain the required result.
\end{proof}
\begin{lemma}\label{3.lemma32}
For all sufficiently large $\alpha$, all sufficiently small $\varkappa$
and all sufficiently large $n$ one has
\begin{equation}\label{4.2}
N((e^{-n\log n}e^{\alpha n},\infty);T_V)
\leq N((0,\varkappa^{-1});A_n^+).
\end{equation}
\end{lemma}
\begin{proof}
1.
Note that the subspace $z^{n+1}\Fc^2$  is orthogonal to $\Pc_n$ in $\Fc^2$
and therefore these subspaces are linearly independent.
Consider the direct sum of subspaces
\begin{equation*}
\Lc_n
=
\bigl(\Ran E([\varkappa^{-1},\infty),{A_n^+})\bigr) \dot{+}  (z^{n+1}\Fc^2)
\end{equation*}
The codimension of $\Lc_n$ in $\Fc^2$ 
equals $(n+1)-N([\varkappa^{-1},\infty);A_n^+)=N((0,\varkappa^{-1});A_n^+).$

Let $\phi\in\Lc_n$, $\phi=f+g$, $f\in\Ran E([\varkappa^{-1};\infty),{A_n^+})$, $g\in z^{n+1}\Fc^2$,
so $\norm{f}^2_+\le \varkappa\norm{f}^2_-$ and
$\norm{\phi}^2_{\Fc^2}=\norm{f+g}^2_{\Fc^2}
=\norm{f}^2_{\Fc^2}+\|g\|^2_{\Fc^2}\ge\|g\|^2_{\Fc^2}$.
Using the Cauchy-Schwartz inequality and \eqref{3:symbol}, we get
\begin{align*}
(T_{V_+}\phi,\phi)_{\Fc^2}
=
(W_+\phi,\phi)_+
&\leq
\sigma_+\norm{f+g}_+^2
\leq
2\sigma_+\norm{f}^2_+ + 2\sigma_+\norm{g}^2_+;
\\
(T_{V_-}\phi,\phi)_{\Fc^2}
=
(W_-\phi,\phi)_-
&\geq
\tau_-\norm{f+g}_-^2
\\
&\geq \tau_-(\norm{f}_-^2+\norm{g}_-^2-\epsilon\norm{f}_-^2 - \epsilon^{-1}\norm{g}_-^2).
\end{align*}
Combining the above estimates and using $\|f\|^2_+\le \varkappa\|f\|^2_-$, we obtain
\begin{gather*}
(T_{V}\phi,\phi)_{\Fc^2}
\leq
2\sigma_+ \norm{g}_+^2
+
\tau_-(\epsilon^{-1}  - 1)\norm{g}_-^2
-
(\tau_--2\varkappa\sigma_+ -\tau_- \epsilon)\norm{f}_-^2.
\end{gather*}
Now we suppose that $\varkappa$ and $\epsilon$ are chosen sufficiently small
so that $\tau_--2\varkappa\sigma_+-\tau_- \epsilon\geq0$.
Then
$$
(T_{V}\phi,\phi)_{\Fc^2}
\leq
2\sigma_+ \norm{g}_+^2
+
\tau_-(\epsilon^{-1}-1)\norm{g}_-^2
\leq
2\sigma_+\norm{g}_+^2+\tau_-\epsilon^{-1}\norm{g}_-^2.
$$
In this way  we have established the inequality
\begin{equation}
\frac{(T_V\phi,\phi)_{\Fc^2}}{\|\phi\|^2_{\Fc^2}}
\le
C\frac{\|g\|^2_++\|g\|^2_-}{\|g\|^2_{\Fc^2}},
\quad
\forall\phi=f+g\in\Lc_n
\label{c1}
\end{equation}
with $C=\max\{2\sigma_+,\tau_-\epsilon^{-1}\}$.

2.
Suppose that $r$ is chosen so large that
$\Omega_+\cup\Omega_-\subset D_r(0)$.
Then $\|g\|^2_++\|g\|^2_-\le \|g\|^2_{L^2(D_r(0))}$, and, therefore,
by Lemma~\ref{lma3.3}(i) and \eqref{c1},
\begin{equation*}
    \frac{(T_V\phi,\phi)_{\Fc^2}}{\|\phi\|^2_{\Fc^2}}
    \leq  
    e^{-n\log n}e^{\alpha n}, \quad \forall \phi\in\Lc_n
\end{equation*}
for all sufficiently large $n$. 
It follows now by the min-max principle that
\begin{equation*}
    N((e^{-n\log n}e^{\alpha n},\infty);T_V)
    \leq 
    \codim\Lc_n=N((0,\varkappa^{-1});A_n^+),
\end{equation*}
for all sufficiently large $n$,
as required.
\end{proof}

\subsection{Proof of Theorems~\ref{1:theorem1.1} and \ref{1:cor1.2} }
\begin{proof}[Proof of Theorem~\ref{1:theorem1.1}]
1.
Let us prove the asymptotics \eqref{1:RWestimate}.
From Lemmas  \ref{3.lemma31}, \ref{3.lemma32}, by interchanging
$\Omega_+$ and $\Omega_-$, and using \eqref{3.1} we obtain
\begin{gather}\label{4.4}
N((e^{-n\log n}e^{-\alpha n},\infty);-T_V)
\geq 
N((0,\varkappa);A_n^-)
=
N((\varkappa^{-1},\infty);A_n^+)
\\ 
\label{4.5}
N((e^{-n\log n}e^{\alpha n},\infty);-T_V)
\leq  
N((0,\varkappa^{-1});A_n^-)
=
N((\varkappa,\infty);A_n^+).
\end{gather}
Combining \eqref{4.1} and \eqref{4.4}, we obtain
\begin{multline}
N((e^{-n\log n}e^{-\alpha n},\infty);|T_V|)
=
N((e^{-n\log n}e^{-\alpha n},\infty);T_V)
+
N((e^{-n\log n}e^{-\alpha n},\infty);-T_V)
\\
\geq 
N((0,\varkappa);A_n^+)
+
N((\varkappa^{-1},\infty);A_n^+)
=
(n+1)- N([\varkappa,\varkappa^{-1}];A_n^+).
\label{4.6}
\end{multline}
In the same way, combining \eqref{4.2} with \eqref{4.5}, we get
\begin{multline}\label{4.7}
N((e^{-n\log n}e^{\alpha n},\infty);\abs{T_V})
\\
\leq 
N((\varkappa,\infty);A_n^+)
+
N((0,\varkappa^{-1});A_n^+)
=
(n+1)+N((\varkappa,\varkappa^{-1}),A_n^+).
\end{multline}
Estimates \eqref{4.6}, \eqref{4.7}, together with Lemma \ref{Lemma3.1}, yield
\begin{align*}
s_{n-m}(T_V)&\ge e^{-n\log n}e^{-\alpha n},
\\
s_{n+m}(T_V)&\le e^{-n\log n}e^{\alpha n},
\end{align*}
for all sufficiently large $n$ and some fixed $m$.
This proves \eqref{1:RWestimate}.

2.
Let us prove the `$+$' version of the asymptotic estimates \eqref{1:Theor1.1pm}.
In order to do this, we combine
Lemmas~\ref{Lemma3.1} and \ref{3:Prop.2.3} with Lemma~\ref{3.lemma31}.
We obtain that
$$
N((e^{-n\log n}e^{-\alpha n},\infty);{T_V})\geq\delta_+ n+o(n), \quad n\to\infty.
$$
The latter inequality can be re-written as
\begin{equation}
\lambda^+_{\delta_+ n+o(n)}(T_V)\ge e^{-n\log n}e^{-\alpha n},
\quad 
n\to\infty
\label{c2}
\end{equation}
which is equivalent to  the first inequality in \eqref{1:Theor1.1pm}.
Similarly, from Lemmas~\ref{Lemma3.1}, \ref{3:Prop.2.3} and \ref{3.lemma32},
we obtain
\begin{equation}
\lambda^+_{\Delta_+n+o(n)}(T_V)\le e^{-n\log n }e^{\alpha n},
\quad n\to\infty
\label{c3}
\end{equation}
which gives the second inequality in \eqref{1:Theor1.1pm}.
The `$-$' version of the asymptotic estimates \eqref{1:Theor1.1pm}
is obtained by applying the same reasoning to  $T_{-V}$.
\end{proof}

By simple variational considerations we can now show that the estimates
\eqref{1:Theor1.1pm} hold under considerably weaker conditions.
\begin{corollary}\label{Cor.3.6}
Let the symbol $V$ be a real bounded  function
with compact support.
Suppose that for some $\varepsilon>0$ there exists a closed set $K$
of positive measure such that $V(z)\ge \varepsilon$ on $K$
and $\Co(K)\cap \Co(\supp V_-)=\varnothing$.
Then for the \emph{positive} eigenvalues of $T_V$ the estimate
\begin{equation*}
    \lambda_n^+(T_V)\ge e^{-\beta n\log n}
\end{equation*}
holds true
for some $\beta<\infty$
for all sufficiently large $n$.
\end{corollary}
\begin{proof}
Let us compare $T_V$ with the Toeplitz operator $T_F$,
where $F(z)=\varepsilon$ on $K$, $F(z)=\inf V$ on $\supp V_-$, 
and $F(z)=0$ elsewhere.
Since $F\le V$, we have $T_F\le T_V$.
Applying Theorem~\ref{1:theorem1.1} to $T_F$, we get the required result
for $T_V$ by variational considerations.
\end{proof}

\begin{proof}[Proof of Theorem~\ref{1:cor1.2}]
By symmetry, $A_n^+$ and $A_n^-$ are unitarily equivalent.
Using this and \eqref{3.1}, we obtain
$$
N((0,1);A_n^+)=N((0,1);A_n^-)=N((1,\infty);A_n^+),
$$
and therefore, by Lemma~\ref{Lemma3.1},
$$
2N((0,1);A_n^\pm)
=
n-N(\{1\};A_n^+)
=
n+O(1), \quad n\to\infty,
$$
and so $N((0,1);A_n^\pm)=\frac{n}{2}+O(1)$.
From here, following the same reasoning as in the second part
of the proof of Theorem~\ref{1:theorem1.1}, we obtain the
required estimates \eqref{1:Cor.1}.
\end{proof}

\section{Applications to the perturbed Landau Hamiltonian}\label{sect.4}
\subsection{The Landau Hamiltonian}
The unperturbed Landau Hamiltonian $H_0$ is the 
Schr\"o\-dinger operator with a  constant homogeneous magnetic field $B$:
\begin{equation}\label{1:Landau}
H_0
=
\bigl(-i\tfrac{\partial }{\partial x_1}+\tfrac{B}2 x_2\bigr)^2
+
\bigl(-i\tfrac{\partial }{\partial x_2}-\tfrac{B}2 x_1\bigr)^2-B
\quad \text{ in }
L^2(\R^2,dx_1 dx_2).
\end{equation}
Here the constant term $-B$ is included
for normalization reasons.
In what follows we take $B=2$; the general case can always 
be reduced to this one by scaling. 
It is well known that the spectrum of $H_0$
consists of the eigenvalues 0, 2, 4,\dots of infinite
multiplicity; these eigenvalues are called the Landau
levels.

Next, let $V:\R^2\to\R$, $V(x)\to 0$ as $\abs{x}\to
\infty$. Consider the operator $H=H_0+V$; here $V$
stands for the operator of multiplication by the
function $V=V(x)$ and has the meaning of the
perturbing electric potential. As it is well known, 
the operator $V$ is $H_0$-form compact and so
the essential spectra of $H_0$ and $H$ coincide.  
Thus, the 
spectrum of $H$ consists of eigenvalues, and the only
possible points of accumulation of these eigenvalues
are the Landau levels. One can say that the
perturbation `splits' the Landau levels into
`clusters' of eigenvalues.
Here we discuss the rate of accumulation of the eigenvalues
in these clusters towards the
Landau levels. For simplicity, we discuss only
the case of the lowest Landau level 0.
Without going into details we mention that combining the technique
of \cite{FilPush} with the technique of this paper, it is not difficult to obtain
similar results for the splitting of the higher Landau levels.

\subsection{Reduction to $P_0 V^\pm P_0$}
For $n\geq0$, let $P_n$ be the spectral projection of $H_0$ corresponding 
to the Landau level $2n$, i.e. $P_n=E(\{2n\};H_0)$.
It is common wisdom that the asymptotic behavior of the eigenvalues of
$H$  in the cluster around the Landau
level $2n$ is determined by the asymptotics of the eigenvalues
of the operator $P_nVP_n$. 
Results of this kind have been widely used in the case of a sign-definite $V$,
see  e.g. \cite{Raikov1,IwaTam1,RaiWar,MelRoz,FilPush,RozTa}.
Our situation, with $V$ changing sign, is somewhat more complicated.
The Theorem to follow expresses the above relation in exact terms.
As mentioned above, we only discuss the case $n=0$.
We fix  $\e>0$ and denote $V^{\pm}_\epsilon=V\pm \e |V|$.
\begin{theorem}\label{theorem5.1} 
(i) There exists a constant $m$ (which may depend on $V$ and on $\e$) 
such that for  any $\lambda<0$,
\begin{equation}\label{5.1}
N((-\infty,\lambda);P_0VP_0)
\le 
N((-\infty,\lambda);H)\le  N((-\infty,\lambda);P_0V_\e^-P_0)+m.
\end{equation}
(ii) For any $a\in (0,2)$, there exists a constant $m$ (which may depend on 
$V$, $\e$ and on $a$), such that for any $\lambda\in(0,a)$
\begin{equation}\label{5.2}
N((\lambda,\infty);P_0V_\e^-P_0)-m
\leq N((\lambda,a),H)\le N((\lambda,\infty);P_0V_\e^+P_0)+m.
\end{equation}
\end{theorem}
Similar  results were obtained by
different methods  in  \cite{Raikov1,RaiWar,MelRoz}.

The key step in the proof of Theorem \ref{theorem5.1} is the following lemma. 
The idea of this lemma is borrowed from \cite{IwaTam1}. 
\begin{lemma}\label{Lemma5.2}
The following operator inequalities hold in the quadratic form sense:
\begin{equation*}
H_\e^-\le H\le H_\e^+,
\end{equation*}
where
\begin{equation*}
H_\e^\pm=H_0+P_0(V\pm\e|V|)P_0+P_0^\bot(V\pm\e^{-1} |V|)P_0^\bot,
\end{equation*}
and $P_0^\bot=I-P_0$.
\end{lemma}
\begin{proof}
This is a direct calculation. 
Let $\psi\in L^2(\R^2), \psi_0=P_0\psi_0, \ \psi_1=P_0^\bot\psi_0$. 
Then
\begin{align*}
(V\psi,\psi)
&\leq
(V\psi_0,\p_0)+(V\p_1,\p_1)+2|(V\p_0,\p_1)|;
\\
2|(V\p_0,\p_1)|
&\leq 
2\norm{\abs{V}^{1/2}\psi_0}\norm{\abs{V}^{1/2}\psi_1}
\le 
\e \||V|^{\frac12}\p_0\|^2
+
\e^{-1}\||V|^{\frac12}\p_1\|^2
\\
&=\e(|V|\p_0,\p_0)+\e^{-1}(|V|\p_1,\p_1),
\end{align*}
and this proves the upper bound. The lower bound is proved in the same way.
 \end{proof}
 \begin{proof}[Proof of Theorem~\ref{theorem5.1}(i)]
 The lower bound in \eqref{5.1} follows trivially from the min-max principle:
 \begin{equation*}
N((-\infty,\lambda);H)
\geq 
N((-\infty,\lambda);P_0HP_0)
=
N((-\infty,\lambda);P_0VP_0),
\end{equation*}
since $P_0H_0P_0=0.$

To prove the upper estimate, we note that by the min-max principle,
\begin{equation*}
H_\e^-\leq H\quad\Rightarrow\quad
N((-\infty,\lambda);H)
\leq  
N((-\infty,\lambda);H_\e^-),
\end{equation*}
and
\begin{multline*}
N((-\infty,\lambda);H_\e^-)
=
N((-\infty,\lambda);P_0V_\e^-P_0)
+
N((-\infty,\lambda);P_0^\perp(H_0+V-\e^{-1}|V|)P_0^\bot)
\\
\leq 
N((-\infty,\lambda);P_0V_\e^-P_0)+m,
\end{multline*}
where 
$m=N((-\infty,0);P_0^\perp(H_0+V-\e^{-1}|V|)P_0^\bot)$. 
Note that the quantity $m$ is finite. 
Indeed,  the bottom of the essential spectrum of the operator
$P_0^\bot H_0P_0^\bot|_{\Ran P_0^\bot}$ is $2$, and 
$P_0^\bot(V-\e^{-1}|V|)P_0^\bot$ is a relatively compact perturbation of 
this operator. 
Thus, the  operator 
$P_0^\perp(H_0+V-\e^{-1}|V|)P_0^\bot$ 
has finitely many negative eigenvalues.
 \end{proof}

In order to prove the second part of Theorem \ref{theorem5.1}, 
we will use some machinery from \cite{Pushn}. 
The argument below has a general operator theoretic nature and 
applies to any pair of self-adjoint operators $H_0$, $H$ which 
are semi-bounded from below and such that the difference
$H-H_0$ is $H_0$-form compact. 
Under this assumption, the essential spectra of $H_0$ and $H$
coincide.
Consider
the difference between the eigenvalue counting functions of $H$
and $H_0$:
\begin{equation}
N((-\infty,\lambda);H_0)-N((-\infty,\lambda);H).
\label{diff}
\end{equation}
Of course, this difference only makes sense for $\lambda<\inf\sigma_\ess(H_0)$; 
if the interval $(-\infty,\lambda)$ contains points of the essential spectrum
of $H_0$ and $H$, then we
formally obtain $\infty-\infty$. 
However, there is a natural regularisation of the difference \eqref{diff} which 
is well defined for all $\lambda\in\R\setminus\sigma_\ess(H_0)$. 
This regularisation is given by 
$$
\Xi(\lambda;H,H_0)=\Index(E((-\infty,\lambda);H_0),E((-\infty,\lambda);H)),
$$
where the r.h.s. is the Fredholm index of a pair of projections
(see \cite{ASS} for the definition and a survey of the index of a pair of projections).
The index $\Xi(\lambda;H,H_0)$ appeared in spectral theory in 
various guises, mostly in connection with the spectral shift function 
theory, see e.g. \cite{GNM,GM,BPR,Hempel}. It was systematically 
surveyed and studied in \cite{Pushn}. The properties of $\Xi(\lambda;H,H_0)$ 
relevant to us are monotonicity,
\begin{equation}\label{5.4}
H^-\le H\le H^+
\quad\Rightarrow\quad 
\Xi(\lambda;H^-,H_0)
\leq
\Xi(\lambda;H,H_0)
\leq
\Xi(\lambda;H^+,H_0),
\end{equation}
and the connection with the eigenvalue counting function,
\begin{equation}\label{5.5}
\Xi(\lambda_1;H,H_0)-\Xi(\lambda_2;H,H_0)
=
N([\lambda_1,\lambda_2);H)-N([\lambda_1,\lambda_2);H_0),
\end{equation}
for any  interval $[\lambda_1,\lambda_2]$ which does not contain points of essential 
spectrum of $H_0$ and $H$.
In particular, taking $\lambda_1=-\infty$, we see that $\Xi(\lambda;H,H_0)$ 
coincides with the difference \eqref{diff} whenever the latter makes sense.
\begin{proof}[Proof of Theorem~\ref{theorem5.1}(ii)]
Let us prove the upper bound in \eqref{5.2}. 
Using \eqref{5.4} and \eqref{5.5}, we obtain
\begin{gather*}
N([\lambda,a);H)= \Xi(\lambda;H,H_0)-\Xi(a;H,H_0)
\leq 
\Xi(\lambda;H_\e^+,H_0)-\Xi(a;H,H_0)
\\
= 
N([\lambda,a);H_\e^+)+\Xi(a;H_\e^+,H_0)-\Xi(a;H,H_0)=N([\lambda,a);H_\e^+)+k_+,
\end{gather*}
for any $\lambda\in(0,a)$, where $k_+$ is independent of $\lambda$.  
It follows that
\begin{equation}\label{5.6}
N((\lambda,a);H)
\leq 
N((\lambda,a);H_\e^+)+k_+, 
\quad\forall \lambda\in(0,a).
\end{equation}
Next,
\begin{multline}\label{5.7}
N((\lambda,a);H_\e^+)
=
N((\lambda,a);P_0V_\e^+P_0)
+
N((\lambda,a);P_0^\bot(H_0+V_\e^+) P_0^\bot)
\\
\leq 
N((\lambda,a);P_0V_\e^+P_0)+N((0,a);P_0^\bot(H_0+V_\e^+) P_0^\bot)
\leq
N((\lambda,\infty);P_0V_\e^+P_0)+l_+,
\end{multline}
where $l_+=N((0,a);P_0^\bot(H_0+V_\e^+) P_0^\bot)$.
Combining \eqref{5.6} and \eqref{5.7}, we obtain the upper bound 
in \eqref{5.2} with $m= l_+ + k_+$.

Next, in order to prove the lower bound in \eqref{5.2},
similarly to \eqref{5.6}, \eqref{5.7}, we obtain
\begin{equation*}
    N((\lambda,a);H)\ge N((\lambda,a);H_\e^-)-k_-, 
    \quad \forall \lambda\in(0,a),
\end{equation*}
for some $k_-<\infty$, and
\begin{multline*}
N((\lambda,a);H_\e^-)
=
N((\lambda,a);P_0V_\e^-P_0)+N((\lambda,a);P_0^\bot(H_0+V_\e^-) P_0^\bot)
\\
\ge 
N((\lambda,a);P_0V_\e^-P_0)
= 
N((\lambda,\infty);P_0V_\e^-P_0)- N([a,\infty);P_0V_\e^-P_0).
\end{multline*}
Combining the last two estimates, 
we  obtain the lower bound in \eqref{5.2} with 
$m=k_-+N([a,\infty);P_0V_\e^-P_0)$.
\end{proof}

\subsection{Connection between $P_0 V P_0$ and $T_V$}
Recall (see e.g. \cite[Section~4]{FilPush})
that the operator $P_0VP_0$ is unitarily equivalent to the
Toeplitz operator $T_V$; in particular, 
\begin{equation}
N((\lambda,\infty);\pm T_V)=N((\lambda,\infty);\pm P_0VP_0),
\quad \forall \lambda>0.
\label{unitary}
\end{equation}
Thus, we can apply the results of the previous Sections
to the study of the splitting of the lowest Landau level.
Suppose that $V$ is of the form \eqref{1:symbol V},
where the closed convex hulls of the Lipschitz domains $\Omega_-$ and $\Omega_+$ are
disjoint and
the functions $v_\pm>0$ are bounded and separated from zero.
Then by choosing $\e>0$ sufficiently small, 
we can ensure that $V_\e^\pm=V\pm\e|V|$ have the form
\begin{equation*}
V^\pm_\e=v_+^{(\pm)}\chi_{\Omega_+}-v_-^{(\pm)}\chi_{\Omega_-}
\end{equation*}
with  $v_+^{(\pm)}$, $v_-^{(\pm)}$ also
bounded and separated from zero. 
Thus,  Theorem \ref{1:theorem1.1} applies to $T_{V^\pm_\e}$. 
Combining this with \eqref{unitary} and 
using Theorem~\ref{theorem5.1}, 
we arrive at the following result.
\begin{theorem}
Under the above assumptions on the potential $V$,
the lowest Landau level necessarily generates two infinite sequences
of eigenvalues of $H$ converging to zero.
For
any $a\in(0,2)$ we have
\begin{equation}
\label{Th.4.3.1}
N((-\infty,-\lambda);H)+N((\lambda,a);H)
=
\frac{|\log \lambda |}{\log|\log \lambda |}(1+o(1)), \quad \lambda\to +0.
\end{equation}
For any $a\in(0,2)$ and some $c>0$ we have
\begin{equation}
\label{Th.4.3.pm}
N((-\infty,-\lambda);H)
\geq
c\frac{|\log \lambda|}{\log|\log \lambda|},
\quad
N((\lambda,a);H)
\geq
c\frac{\abs{\log \lambda}}{\log\abs{\log \lambda}}
\end{equation}
for all sufficiently small $\lambda>0$.
\end{theorem}

Following the same pattern, one can also apply
Theorems~\ref{5:Teo2}, \ref{Theorem2G} and Corollary~\ref{Cor.3.6}
to the analysis of the eigenvalues of $H$.

\section*{Acknowledgements}
A.P. is grateful to The Swedish Royal Academy of Sciences 
for the financial support  and to the 
Department of Mathematics, Chalmers University of Technology,
for hospitality. 
G.R. is grateful to the London Mathematical Society for the 
financial support (Scheme 2 grant no. 2804)
and to the Department of Mathematics, King's College London
for hospitality.
Both authors are grateful to Nikolai Filonov for reading the manuscript
and making a number of useful remarks.

\end{document}